\documentclass[a4paper]{amsart}
\usepackage[utf8]{inputenc}
\usepackage{array,epsfig}
\usepackage{amsmath}
\usepackage{amsfonts}
\usepackage{amssymb}
\usepackage{amsxtra}
\usepackage{amsthm}
\usepackage{mathrsfs}
\usepackage{color}
\newcommand{\N}{\mathbb{N}}

\newcommand{\grad}{\text{grad}}

\newcommand{\Ric}{\textbf{Ric}}

\newenvironment{lemma}[2][Lemma]{\begin{trivlist}
\item[\hskip \labelsep {\bfseries #1}\hskip \labelsep {\bfseries #2.}]}{\end{trivlist}}

\newcommand{\vol}{\textnormal{vol}}

\newcommand{\ra}{\rightarrow}

\newcommand{\lb}{\langle}
\newcommand{\rb}{\rangle}

\newcommand{\RNum}[1]{\uppercase\expandafter{\romannumeral #1}}
\newtheorem{thm}{Theorem}[section]

\theoremstyle{definition}

\newtheorem{prop}{Proposition}[section]

\theoremstyle{remark}

\begin{document}
\title{Scalar Curvature Volume Comparison Theorems for Almost Rigid Sphere}
\author{Y.Zhang}
\begin{abstract}
    Bray's football theorem (\cite{bray2009penrose}) is a weakening of Bishop theorem in dimension 3. It gives a sharp volume upper bound for a three dimensional manifold with scalar curvature larger than $n(n-1)$ and Ricci curvature larger than $\varepsilon$.  This paper extends Bray's football theorem in high dimensions, assuming the manifold is axis symmetric or the Ricci curvature has an upper bound.
\end{abstract}
\maketitle
\section{Introduction}

Bishop theorem is a classical theorem in differential geometry that establishes the connection between volume and Ricci curvature. It was proven by Bishop in 1963 \cite{bishop}.

We assume throughout that $(M,g)$ is a compact smooth n-dimensional Riemannian manifold. Let $(S^n,\Bar{g})$ be the unit n-sphere with standard metric, i.e. it has constant sectional curvature 1. Let $\Ric_g$, $R_g$ and $\vol(M)$ be the Ricci curvature, scalar curvature and volume of $(M,g)$, respectively.
\begin{thm}[Bishop theorem]
 If $\Ric_g\ge (n-1)g$, then $\vol(M)\le \vol(S^n)$.
\end{thm}

A classical approach to prove this theorem is using geodesic balls  (see \cite{petersen2016}). Another approach is from optimal transport in Lott, Villani and Strum's seminal papers (see \cite{lott2009ricci},\cite{sturm1},\cite{sturm2}). They defined a synthetic Ricci curvature on metric measure spaces using optimal transport. Thus  Bishop theorem can be generalized to metric measure spaces. The third approach was discovered by H.Bray in his thesis using isoperimetric surfaces (\cite{bray2009penrose}). A byproduct of the third approach is Bray's football theorem which is a weakening of Bishop theorem in dimension 3.

\begin{thm}[Bray's football theorem] If $(M,g)$ satisfies: 
$\Ric_g\ge \varepsilon(n-1)g$, $R_g\ge n(n-1)$, $\varepsilon\in(0,1)$, then $V_g\le \alpha(\varepsilon)V_{\bar{g}}$, where $\alpha(\varepsilon)=1$, when $\varepsilon\in[\varepsilon_0,1)$; $\alpha(\varepsilon)>1$,  when $\varepsilon\in(0,\varepsilon_0)$.
\end{thm}
For a full expression of $\alpha(\varepsilon)$, the readers can find it in (\cite{bray2009penrose},\cite{jeffv} and \cite{bray2019proof}). Regarding the constant $\varepsilon_0$ in Theorem 1.2, the numerical results show $0.134<\varepsilon_0<0.135$; M. Gurskya and J. Viaclovskyb proved $\varepsilon_0\le 0.5$.
When $\varepsilon\in(0,\varepsilon_0)$, $(M,g)$ with the largest volume in Theorem 1.2 is axis symmetric, i.e. $(M,g)$ has the shape of a football (American football). For the case of axis symmetry in higher dimensions, we have the following theorem:

\begin{thm}
  Let $n\ge 3$. If $(M,g)$ is axis symmetric, i.e. $M=[0,a]\times_f S^{n-1}$,  $g=dt^2+f(t)^2d\sigma^2$, where $d\sigma^2$ is the standard metric  of $S^{n-1}$. There exists an $\varepsilon (n)< 1$, such that for any axis symmetric manifold $(M,g)$ satisfies:
    \[\Ric(g)\ge \varepsilon(n)\cdot \Ric_0\cdot g\ \textnormal{and}\  R(g)\ge R_0, \]
    we have: $\vol(M)\le \vol (S^n)$.
    \end{thm}

If we assume the manifold has a uniform upper bound for Ricci curvature, we have Theorem 1.4 which is a high dimensional analog of Bray's football theorem. 
\begin{thm}
For any $C>0$, there exists an $\varepsilon=\varepsilon(n,C)\in(0,1)$, such that
for any compact Riemannian manifold  $(M,g)$ satisfies
\begin{enumerate}
    \item $(1-
    \varepsilon)(n-1)g\le\Ric_g\le Cg$,
    \item $R_g\ge n(n-1)$,
\end{enumerate}
then $\vol(M)\le \vol(S^n)$.
\end{thm}
If we choose $\varepsilon$ sufficiently small, then the results in (\cite{cheeger1997structure}) show that $M$ is diffeomorphic to $S^n$.  Then   
according to Andersen's paper (\cite{anderson1990convergence}), $g$ is close to $\bar{g}$ in $C^{1,\alpha}$ norm.  Hence, Theorem 1.3 is a directly result of our main Theorem 3.1.
We use the tools in \cite{brendle2011} to prove the main Theorem 3.1. The perturbation formula of the scalar curvature is crucial for deriving the contradiction. Theorem 1.4 is slightly stronger than Corollary A in \cite{yuan2016volume}, while  Corollary A needs the metric $g$ on $S^n$ is close to the standard metric in $W^{2,p}$ norm. 
\section*{Acknowledgement}
The author is grateful to Hubert Bray for the helpful advice on this paper. Thanks also to Simon Brendle for valuable explanations and communications related to this work. 

\section{Proof of theorem 1.3} 

Since $g=dt^2+f(t)^2d\sigma^2$, according to O'Neil\cite{o1983semi}, for vertical tangent vectors $v,w\in i_*(TS^{n-1})$, \[Ric_g(v,w)=Ric(v,w)^{S^{n-1}}-\langle v,w\rangle f^\#,\]
\[\textnormal{where}\  f^{\#}=\frac{\Delta f}{f}+(n-2)\frac{\langle \grad f,\grad f\rangle}{f^2}.\]
\[Ric_g(\partial_t,\partial_t)=-\frac{n-1}{f}H^f(\partial_t,\partial_t).\]

Since $0,a$ are two end points, we have $f(0)=f(a)=0$.

Without loss of generality, we assume 
$f(t)\ge0$, $f'(0)\ge  0$, then the curvature conditions in Theorem 1.3 imply:  
    \begin{equation}
       -\frac{f''}{f}\ge \varepsilon ,
    \end{equation}
    \begin{equation}
       \frac{n-2}{f^2}-\frac{f''}{f}-(n-2)\frac{(f')^2}{f^2}\ge (n-1)\varepsilon,
    \end{equation}
    \begin{equation}
    -\frac{2f''}{f}+\frac{n-2}{f^2}-(n-2)\frac{(f')^2}{f^2}\ge n,
\end{equation}
where $\varepsilon\in(0,1]$. 

Assume $C_\varepsilon(f)=\vol(M)/\vol(S^n)$, then:
\begin{equation*}
    \displaystyle{C_\varepsilon(f)=\frac{\omega_{n-1}\int_0^s f(t)^{n-1} dt}{\omega_n}},
\end{equation*}
where $\omega_n$ is the volume of $S^n$. 

Note when  $f=\sin t$, $(M,g)$ is a n-sphere with standard metric and $f$ satisfies equations (1)-(3),   $C_{\varepsilon}(\sin t)=1$. 

Hence, we need to prove: if $\varepsilon$ is  sightly less that 1, we still have $C_\varepsilon(f)\le1$.

Before proving Theorem 1.3, we need to prove a lemma first:

\begin{lemma}{1}
  For $x\in [0.5,\infty)$, \[\sqrt{x-0.5}\cdot \frac{\Gamma(x)}{\Gamma(x+0.5)}< \sqrt{x+0.5}\cdot\frac{\Gamma(x+1)}{\Gamma(x+1.5)}.\]
\end{lemma}
\begin{proof}
Since $\Gamma(x+1)=x\Gamma(x)$, therefore, 
\begin{align*}
    \sqrt{x-0.5}\cdot \frac{\Gamma(x)}{\Gamma(x+0.5)}&=\sqrt{x-0.5}\cdot\frac{\Gamma(x+1)}{\Gamma(x+1.5)}\cdot\frac{x+0.5}{x}
    \\&\le \sqrt{x+0.5}\cdot\frac{\Gamma(x+1)}{\Gamma(x+1.5)}.
\end{align*}
\end{proof}

In fact, $\sqrt{x-0.5}\cdot \frac{\Gamma(x)}{\Gamma(x+0.5)}$ is increasing on $[0.5,\infty)$, however, lemma 1 is enough for us to prove Theorem 1.3.

 Equation (1) implies $f(t)$ is concave, 
we can assume $f'(t)\ge 0$, for $t\in[0,r]$, where  $r$ satisfies $f'(r)=0$. We only focus on $[0,r]$, since the case of $[r,a]$ is similar to $[0,r]$ by symmetry. 
Equation (3) implies: $f^{n-2}(1-(f')^2-f^2)$ is increasing. Hence, $f^2+f'^2\le 1$, then we have: equation (1) and (3) imply equation (2).

Assume $m=f(r)$, i.e. $m=\max_{t\in[0,a]} f(t)$,\ then for $0<t\le r$,  $f^{n-2}(1-(f')^2-f^2)\le m^{n-2}(1-m^2)$, so $f'\ge (1-f^2-\frac{m^{n-2}(1-m^2)}{f^{n-2}})^{1/2}$.

Equation (1) implies: $\varepsilon f^2+(f')^2$ is decreasing, so $f'\ge (\varepsilon (m^2-f^2))^{1/2}$.

Therefore, we have two lower bounds for $f'$, let $M_\varepsilon(f)$ be:   \[M_\varepsilon(f)=\max\{1-f^2-\frac{m^{n-2}(1-m^2)}{f^{n-2}}, \varepsilon (m^2-f^2)\}.\]

Then we have:
\begin{equation}
    \int_0^r f(t)^{n-1}dt=\int^m_0 \frac{f^{n-1}}{f'}df\le \int^m_0 \frac{f^{n-1}}{M_\varepsilon(f)^{1/2}}df.
\end{equation}

We substitute $f$ by $ms$, then: $s\in[0,1)$, 
\begin{align*}
    1-f^2-\frac{m^{n-2}(1-m^2)}{f^{n-2}}&=1-m^2s^2-\frac{m^{n-2}(1-m^2)}{m^{n-2}s^{n-2}}
    \\&=m^2(1-s^2)\left[1-\frac{(1-m^2)(1-s^{n-2})}{m^2s^{n-2}(1-s^2)}\right],
\end{align*}
\begin{equation*}
    \int^m_0\frac{f^{n-1}}{M_\varepsilon(f)^{1/2}}df=\int^1_0\frac{m^{n-1}s^{n-1}ds}{\sqrt{\max\{(1-s^2)[1-\frac{(1-m^2)(1-s^{n-2})}{m^2s^{n-2}(1-s^2)}],\varepsilon (1-s^2)\}}}.
\end{equation*}

If $m^{n-1}\le \varepsilon^{1/2}$, then 
\[ \int^m_0\frac{f^{n-1}}{M_\varepsilon(f)^{1/2}}df\le \int^1_0\frac{m^{n-1}t^{n-1}}{[\varepsilon(1-t^2)]^{1/2}}\le \int^1_0 \frac{t^{n-1}}{(1-t^2)^{1/2}},\]
i.e. $\vol(M)\le \vol(S^n)$. Therefore, we can assume $m^{n-1}\ge \varepsilon^{1/2}$. 

Since when $s\ra 1_-$, we have $1-\frac{(1-m^2)(1-s^{n-2})}{m^2s^{n-2}(1-s^2)}\ra 1-\frac{(n-2)(1-m^2)}{2m^2}$,  then
\[1-\frac{(n-2)(1-m^2)}{2m^2}-\varepsilon\ge 1-\frac{(n-2)(1-m^2)}{2m^2}-m^{2(n-1)}=(1-m^2)(\sum^{n-2}_{i=0} m^{2i}-\frac{n-2}{2m^2}). \]

Therefore, there exists $\varepsilon\in(0,1)$ such that $1-\frac{(n-2)(1-m^2)}{2m^2}> \varepsilon$, for any $\varepsilon^{\frac{1}{2(n-2)}}\le m\le 1$. Then, we can define: 
\[h(m)=\max\{x|x\in (0,1), x\  \textnormal{satisfies}\ 1-\frac{(1-m^2)(1-x^{n-2})}{m^2x^{n-2}(1-x^2)}=\varepsilon \}.\]
Hence, for any $s\in(h(m),1)$, 
\[1-\frac{(1-m^2)(1-s^{n-2})}{m^2s^{n-2}(1-s^2)}>\varepsilon.\]

Assume $\displaystyle{H(m)=\int^{h(m)}_0 \frac{m^{n-1}t^{n-1}}{[\varepsilon(1-t^2)]^{1/2}}dt+\int_{h(m)}^1 \frac{m^{n-1}t^{n-1}}{[(1-t^2)(1-\frac{(1-m^2)(1-t^{n-2})}{m^2t^{n-2}(1-t^2)})]^{1/2}}dt}$. 

Then we have: \[H(m)\ge\int^m_0\frac{f^{n-1}}{M_\varepsilon(f)^{1/2}}df \ge\int_0^r f(t)^{n-1}dt.\]

Since when $m\ra 1$, $h(m)\ra 0$, then the expression of $H(1)$ is exactly the volume of hemisphere. As a result of this observation, we need to show $H(m)\le H(1)$, so we estimate $H'(m)$ as shown below.


\begin{align*}
     m^{4-n}H'(m)=&(n-1)\Bigg[\int^1_{h(m)}\frac{m^2t^{n-1}dt}{\sqrt{(1-t^2)(1-\frac{(1-m^2)(1-t^{n-2})}{m^2t^{n-2}(1-t^2)})}}+\int^{h(m)}_0 \frac{m^2t^{n-1}dt}{[\varepsilon(1-t^2)]^{1/2}}\Bigg]
 \\&-\int^1_{h(m)}\frac{t(1-t^{n-2})dt}{(1-t^2)^{3/2}(1-\frac{(1-m^2)(1-t^{n-2})}{m^2t^{n-2}(1-t^2)})^{3/2}}
 \\ \ge &(n-1) m^2 \int^1_0\frac{t^{n-1}dt}{(1-t^2)^{1/2}}-\varepsilon^{-3/2}\int^1_0\frac{t}{(1-t^2)^{1/2}}\cdot \frac{1-t^{n-2}}{1-t^2}dt
 \\ \ge &(n-1)\varepsilon^{1/(n-1)}\int^1_0\frac{t^{n-1}dt}{(1-t^2)^{1/2}}-\varepsilon^{-3/2}\int^1_0\frac{t}{(1-t^2)^{1/2}}\cdot \frac{1-t^{n-2}}{1-t^2}dt. 
\end{align*}


The rest part of this section is to prove when $n\ge 4$,  
\begin{equation} \label{gamma1}
    (n-1)\int^1_0\frac{t^{n-1}dt}{(1-t^2)^{1/2}}>\int^1_0\frac{t}{(1-t^2)^{1/2}}\cdot \frac{1-t^{n-2}}{1-t^2}dt.
\end{equation}

If inequality (\ref{gamma1}) holds, then we can find an $\varepsilon<1$
 such that $H'(m)> 0$, for $m\in [\varepsilon^{\frac{1}{2(n-1)}},1]$.
 

We divide $n\ge 4$ into two situations: $n=2k+2$ and $n=2k+1$, $k\in \N$, as they are slightly different.
\subsection{$n=2k+2$, $k\ge 1$:}
\[(n-1)\int^1_0\frac{t^{n-1}}{(1-t^2)^{1/2}}dt=(2k+1)\int^1_0\frac{t^{2k+1}}{(1-t^2)^{1/2}}dt=(2k+1)\frac{\sqrt{\pi}}{2}\frac{\Gamma(k+1)}{\Gamma(k+\frac{3}{2})},\]
\[\int^1_0\frac{t}{(1-t^2)^{1/2}}\cdot \frac{1-t^{n-2}}{1-t^2}dt=\frac{\sqrt{\pi}}{2}\sum^{k-1}_{j=0}\frac{\Gamma(j+1)}{\Gamma(j+\frac{3}{2})}.\]
According to lemma 1, to prove equation (\ref{gamma1}), we need to show 
\[\frac{(2k+1)}{\sqrt{k+0.5}}\ge \sum^{k-1}_{j=0}\frac{1}{\sqrt{j+1/2}}.\]

As $\frac{1}{\sqrt{j+1/2}}\le \frac{2}{\sqrt{j+1}+\sqrt{j}}=2\sqrt{j+1}-2\sqrt{j}$, so $\displaystyle{\sum^{k-1}_{j=0}\frac{1}{\sqrt{j+1/2}}\le 2\sqrt{k}\le \frac{(2k+1)}{\sqrt{k+0.5}}}.$

\subsection{$n=2k+1$, $k\ge 2$:}
\[(n-1)\int^1_0\frac{t^{n-1}}{(1-t^2)^{1/2}}dt=2k\int^1_0\frac{t^{2k}}{(1-t^2)^{1/2}}dt=2k\frac{\sqrt{\pi}}{2}\frac{\Gamma(k+1/2)}{\Gamma(k+1)},\]
\begin{align*}
    \int^1_0\frac{t}{(1-t^2)^{1/2}}\cdot \frac{1-t^{n-2}}{1-t^2}dt=&\int^1_0\frac{t^2}{\sqrt{1-t^2}}(1+t^2+\cdots+t^{2k-4})+\frac{t}{(t+1)(1-t^2)^{1/2}}dt
\\ =&\frac{\sqrt{\pi}}{2}\sum^{k-1}_{j=1}\frac{\Gamma(j+1/2)}{\Gamma(j+1)}+\frac{\pi-2}{2}.
\end{align*}
We need to prove: $\displaystyle{k\sqrt{\pi}\frac{\Gamma(k+1/2)}{\Gamma(k+1)}\ge \frac{\sqrt{\pi}}{2}\sum^{k-1}_{j=1}\frac{\Gamma(j+1/2)}{\Gamma(j+1)}+\frac{\pi-2}{2}}$,
\\ dividing by $\frac{\sqrt{\pi}}{2}\cdot\sqrt{k}\cdot\frac{\Gamma(k+1/2)}{\Gamma(k+1)}$ at both sides, applying lemma 1 , all we need to show is:
\[2\sqrt{k}\ge \sum^{k-1}_{j=1}\frac{1}{\sqrt{j}}+\frac{(\pi-2)/\sqrt{\pi}}{\sqrt{2}\cdot\frac{\Gamma(2.5)}{\Gamma(3)}}.\]
As $\displaystyle{\frac{(\pi-2)/\sqrt{\pi}}{\sqrt{2}\cdot\frac{\Gamma(2.5)}{\Gamma(3)}}\approx 0.685\le \sqrt{2}}$,
\\since $\displaystyle{\frac{1}{\sqrt{j}}\le \frac{2}{\sqrt{j+1/2}+\sqrt{j-1/2}}=2\sqrt{j+1/2}-2\sqrt{j-1/2}}$, 
\\then $\displaystyle{\sqrt{2}+\sum^{k-1}_{k=1}\frac{1}{\sqrt{j}}\le 2\sqrt{k}}.
$

\section{Proof of theorem 1.4}
The following proposition is  Proposition 11 in \cite{brendle2011}, while the original proposition is in space $W^{2,p}$.  However, the proof in \cite{brendle2011} can be applied to our circumstance with few modifications, since we can still split a  $W^{1,p}$ symmetric two-tensor into a divergence free two-tensor and a Lie derivative of the metric.
\begin{prop}
Assume $p>n$. $\Omega$ is an n-dimensional compact manifold with boundary. Let $g$, $\bar{g}$ be Riemannian metrics on $\Omega$.  
 If $\|g-\bar{g}\|_{W^{1,p}(\Omega,g)}$ is sufficiently small, there exists a diffeomorphism $\varphi: \Omega\rightarrow \Omega$, such that $\varphi|_{\partial\Omega}=id$ and $h=\varphi^*(g)-\bar{g}$ is divergence free.  
 Moreover, there exists a positive constant $C$ that depends on $\Omega$, such that:
\[\|h\|_{W^{1,p}(\Omega,\bar{g})}\le C\|g-\bar{g}\|_{W^{1,p}(\Omega,\bar{g})}.\]
\end{prop}

\begin{thm}
Assume $(S^n,\Bar{g})$ is the n-sphere with standard metric. Let $g$ be another metric on $S^n$ with the following properties:
\begin{enumerate}
    \item $R_g\ge R_{\Bar{g}}=n(n-1)$,
    \item $V_{g}\ge V_{\Bar{g}}$, 
\end{enumerate}
where $V_g$, $V_{\bar{g}}$ is the volume of $(S^n,g)$ and $(S^n,g)$

If $h=g-\Bar{g}$ is sufficiently small in $W^{1,p}(S^n,\bar{g})$ norm, $p>\frac{n}{2}$, then $V_g=V_{\Bar{g}}$, moreover, there exists a diffeomorphism $\varphi: S^n\ra S^n$, such that $\varphi^*(\Bar{g})=g$.
\end{thm}
\begin{proof}
Proposition 4 in \cite{brendle2011} exhibits a pointwise estimate for  $R_g$:
    \begin{align*}
       |&R_g-R_{\Bar{g}}+\lb \Ric_{\Bar{g}},h\rb-\lb\Ric_{\Bar{g}},h^2\rb+\frac{1}{4}\|\nabla h\|^2-\frac{1}{2}\nabla_i h_{kp}\cdot\nabla_k h_{ip}
       \\&+\frac{1}{4}\|\nabla tr(h)\|^2+\nabla_i(g^{ik}g^{jl}(\nabla_k h_{jl}-\nabla_l h_{jk}))
    |
    \\ &\le C\|h\|\cdot\|\nabla h\|^2+C\|h\|^3,
    \end{align*}
where $\|\cdot\|$ is the pointwise norm under $\bar{g}$, $\nabla$ is the Levi-Civita connection of $\Bar{g}$, $tr(h)$ is the trace of $h$ under metric $\Bar{g}$.
\begin{align*}
    &|\int R_g-R_{\bar{g}}+(n-1)tr(h)-(n-1)\|h\|^2+\frac{1}{4}\|\nabla h\|^2
    \\&+\frac{1}{2} h_{kp}\cdot \nabla_i\nabla_k h_{ip}
    +\frac{1}{4}\|\nabla tr(h)\|^2dV_{\bar{g}}|
    \\ \le& C\int \|h\|\cdot \|\nabla h\|^2+\|h\|^3dV_{\bar{g}}.
\end{align*}

Since $\bar{R}_{ijkl}=\bar{g}_{il}\bar{g}_{jk}-\bar{g}_{ik}\bar{g}_{jl}$, we have:
\begin{align*}
     \nabla_i\nabla_k h_{ip}=&\nabla_k\nabla_i h_{ip}-\bar{R}_{ikim}h_{mp}-\bar{R}_{ikpm}h_{im}
     \\=&\nabla_k\nabla_i h_{ip}-tr(h)\Bar{g}_{kp}+nh_{kp}.
 \end{align*}
 
According to Proposition 3.1, we can assume $\nabla \cdot h=0$, up to a diffeomorphism $\varphi$.
Therefore, 
\begin{align*}
    \int R_g-R_{\bar{g}}dV_{\bar{g}}=&\int -(n-1)tr(h)+(n-1)\|h\|^2-\frac{1}{4}\|\nabla h\|^2
    \\&-\frac{1}{2}h_{kp}\cdot\nabla_i \nabla_k h_{ip}
    -\frac{1}{4}\|\nabla tr(h)\|^2 dV_{\bar{g}}
    \\&+O(\|h\|_{C^0(S^n,\bar{g})}\|h\|^2_{W^{1,2}(S^n,\bar{g})})
    \\=& \int-(n-1)tr(h)+(\frac{n}{2}-1)\|h\|^2-\frac{1}{4}\|\nabla h\|^2
    -\frac{1}{4}\|\nabla tr(h)\|^2
    \\&+\frac{1}{2}tr(h)^2dV_{\bar{g}}+O(\|h\|_{C^0(S^n,\bar{g})}\|h\|^2_{W^{1,2}(S^n,\bar{g})}).
\end{align*}
Since 
\[V_g-V_{\bar{g}}=\int\frac{1}{2}tr(h)+\frac{1}{8}tr(h)^2-\frac{1}{4}\|h\|^2d{V_{\bar{g}}}+O(\|h\|_{C^0(S^n,\bar{g})}\|h\|^2_{L^2(S^n,\bar{g})}).\]
Assume $\delta=\int tr(h)dV_{\bar{g}}/V_{\bar{g}}$, $k=\frac{8(n-1)-4\delta}{4+\delta}$ satisfies \[\frac{k}{2}-(n-1)=-\delta(\frac{1}{2}+\frac{k}{8}),\]  
then:
\begin{align*}
    &\int  R_g-R_{\bar{g}}dV_{\bar{g}}+k(V_g-V_{\bar{g}})
    \\=& \int (\frac{k}{2}-n+1)tr(h)+(\frac{k}{8}+\frac{1}{2})tr(h)^2+(\frac{n}{2}-1-\frac{k}{4})\|h\|^2
    \\&-\frac{1}{4}\|\nabla tr(h)\|^2-\frac{1}{4}\|\nabla h\|^2d{V_{\bar{g}}}+O(\|h\|_{C^1(S^n,\bar{g})}\|h\|^2_{W^{1,2}(S^n,\bar{g})})
    \\=& \int(\frac{k}{8}+\frac{1}{2})(tr(h)-\delta)^2+(\frac{n}{2}-1-\frac{k}{4})\|h\|^2
    \\&-\frac{1}{4}\|\nabla tr(h)\|^2-\frac{1}{4}\|\nabla h\|^2d{V_{\bar{g}}}+O(\|h\|_{C^0(S^n,\bar{g})}\|h\|^2_{W^{1,2}(S^n,\bar{g})}).
\end{align*}

$\|\nabla h\|^2\ge \frac{1}{n}\|\nabla tr(h)\|$. $\|h\|^2\ge \frac{1}{n}tr(h)^2$, then \[\int \|h\|^2d{V_{\bar{g}}}\ge \frac{1}{n}\int (tr(h)^2-\delta^2)dV_{\bar{g}}= \frac{1}{n}\int (tr(h)-\delta)^2dV_{\bar{g}}.\]

Since $\int [tr(h)-\delta] dV_{\bar{g}}=0$, by Poincare inequality, we have $\|\nabla tr(h)\|^2_{L^2}\ge n \|tr(h)-\delta\|^2_{L^2}$.

Let $k=2(n-1)-\varepsilon$, we have $|\varepsilon|\le (n+1)\delta$.

Therefore, we can show:
\begin{align*}
    &\int  R_g-R_{\bar{g}}dV_{\bar{g}}+k(V_g-V_{\bar{g}})
    \\=&\int(\frac{n+1}{4}-\frac{\varepsilon}{8})(tr(h)-\delta)^2-(\frac{1}{2}-\frac{\varepsilon}{4})\|h\|^2
    \\&-\frac{1}{4}\|\nabla tr(h)\|^2-\frac{1}{4}\|\nabla h\|^2d{V_{\bar{g}}}+O(\|h\|_{C^0(S^n,\bar{g})}\|h\|^2_{W^{1,2}(S^n,\bar{g})})
    \\=& \int (\frac{n}{4} (tr(h)-\delta)^2-\frac{1}{4}\|\nabla tr(h)\|^2)+ (\frac{1}{4}-\frac{\varepsilon}{8}-\frac{1}{4n})((tr(h)-\delta)^2-\|\nabla h\|^2)
    \\&+(\frac{1}{4n}(tr(h)-\delta)^2-\frac{1}{4}\|h\|^2)
    -(\frac{1}{4}-\frac{\varepsilon}{4})\|h\|^2-(\frac{1}{4n}+\frac{\varepsilon}{8})\|\nabla h\|^2dV_{\bar{g}}
    \\&+O(\|h\|_{C^0(S^n,\bar{g})}\|h\|^2_{W^{1,2}(S^n,\bar{g})})
    \\ \le & -(\frac{1}{4n}+\frac{\varepsilon}{8})\|h\|^2_{W^{1,2}(S^n,\bar{g})}+O(\|h\|_{C^0(S^n,\bar{g})}\|h\|^2_{W^{1,2}(S^n,\bar{g})})\le0
\end{align*}

So we have $h=0$, $g=\bar{g}$.
\end{proof}
Then combining Theorem 1.2 in \cite{anderson1990convergence}, we have Theorem 1.3.
 
\bibliographystyle{plain}

\begin{thebibliography}{10}

\bibitem{anderson1990convergence}
Michael~T Anderson.
\newblock Convergence and rigidity of manifolds under ricci curvature bounds.
\newblock {\em Inventiones mathematicae}, 102(1):429--445, 1990.

\bibitem{bishop}
Richard Bishop.
\newblock A relation between volume, mean curvature and diameter.
\newblock {\em Notices Amer. Math. Soc}, 10(364):t963, 1963.

\bibitem{bray2019proof}
Hubert Bray, Feng Gui, Zhenhua Liu, and Yiyue Zhang.
\newblock Proof of bishop’s volume comparison theorem using singular soap
  bubbles.
\newblock {\em arXiv preprint arXiv:1903.12317}, 2019.

\bibitem{bray2009penrose}
Hubert~L Bray.
\newblock The penrose inequality in general relativity and volume comparison
  theorems involving scalar curvature (thesis).
\newblock {\em arXiv preprint arXiv:0902.3241}, 2009.

\bibitem{brendle2011}
Simon Brendle and Fernando~C. Marcques.
\newblock Scalar curvature rigidity of geodesic balls in $s^n$.
\newblock {\em J. Differential Geom.}, 88(3):379--394, 07 2011.

\bibitem{cheeger1997structure}
Jeff Cheeger, Tobias~H Colding, et~al.
\newblock On the structure of spaces with ricci curvature bounded below. i.
\newblock {\em Journal of Differential Geometry}, 46(3):406--480, 1997.

\bibitem{jeffv}
Matthew Gurskya and Jeff Viaclovskyb.
\newblock Volume comparison and the $\sigma_k$-yamabe problem.
\newblock {\em Advances in Mathematics 187}, 2004.

\bibitem{lott2009ricci}
John Lott and C{\'e}dric Villani.
\newblock Ricci curvature for metric-measure spaces via optimal transport.
\newblock {\em Annals of Mathematics}, pages 903--991, 2009.

\bibitem{o1983semi}
Barrett O'neill.
\newblock {\em Semi-Riemannian geometry with applications to relativity},
  volume 103.
\newblock Academic press, 1983.

\bibitem{petersen2016}
Peter Petersen.
\newblock {\em Riemannian Geometry}, volume 171.
\newblock Springer, 2016.

\bibitem{sturm1}
Karl-Theodor Sturm et~al.
\newblock On the geometry of metric measure spaces.
\newblock {\em Acta mathematica}, 196(1):65--131, 2006.

\bibitem{sturm2}
Karl-Theodor Sturm et~al.
\newblock On the geometry of metric measure spaces. ii.
\newblock {\em Acta Mathematica}, 196(1):133--177, 2006.

\bibitem{yuan2016volume}
Wei Yuan.
\newblock Volume comparison with respect to scalar curvature.
\newblock {\em arXiv preprint arXiv:1609.08849}, 2016.

\end{thebibliography}

\end{document}